\documentclass[11pt]{amsart}
\usepackage{amssymb, amsmath}
\usepackage{TikZ}

\title{$\mathrm{CS}\mathfrak{X}$-groups and free constructions}
\author{O. Al-Raisi}
\address{O. Al-Raisi: Department of  Mathematics,  College of Science, Sultan Qaboos University, Muscat, Oman}
\email{omartalibmiran@gmail.com}

\author{M. Shahryari}
\address{M. Shahryari: Department of  Mathematics,  College of Science, Sultan Qaboos University, Muscat, Oman}
\email{m.ghalehlar@squ.edu.om}

\newtheorem{lemma}{Lemma}[section]

\numberwithin{equation}{section}

\newcommand{\XT}{\mathfrak{X}\mathrm{T}}
\newcommand{\CSX}{\mathrm{CS}\mathfrak{X}}
\newcommand{\X}{\mathfrak{X}}
\newcommand{\Stab}{\mathrm{Stab}}
\newcommand{\Fix}{\mathrm{Fix}}
\begin{document}

\maketitle
\begin{abstract}
Let $\X$ be a variety of groups. A group $G$ is $\CSX$ if its  maximal $\X$-subgroups are malnormal. We prove that the amalgamated free product of two $\CSX$-groups with a malnormal amalgam and without any involutions is again $\CSX$. Using this, we describe conditions under which  the HNN-extension of a $\CSX$-group  is $\CSX$.
\end{abstract}
\vspace{1cm}


\vspace{1cm}


A subgroup $H$ of a given group $G$ is called {\em malnormal} if for every element $x\in G\setminus H$, we have  $H\cap H^x=\{1\}$. A group $G$ is called {\em conjugate separable abelian} ($\mathrm{CSA}$) if all maximal abelian subgroups of $G$ are malnormal.
The class of $\mathrm{CSA}$-groups is broad and plays a significant role in the study of residually free groups, universal theory of non-abelian free groups, limit groups, exponential groups and equational domains in algebraic geometry over groups (see \cite{BMR}, \cite{Champ}, \cite{MR1} , and \cite{MR2}). Another class of groups with very close connections to $\mathrm{CSA}$-groups is the class of $\mathrm{CT}$-groups ({\em commutative transitive groups}). A group is $\mathrm{CT}$ if commutativity is a transitive relation on the set of its non-identity elements. Despite this simple definition, the class of CT groups is also central to the study of residually free groups. Every $\mathrm{CSA}$-group is $\mathrm{CT}$ but the converse is not true.

The second author of the present paper proposed an idea in \cite{Shah} for generalizing the notions of $\mathrm{CSA}$- and $\mathrm{CT}$-groups: Suppose that $\X$ is a class  of  groups. A group $G$  is an   $\XT$-group  if and only if, for any two $\X$-subgroups $A$ and $B$ of $G$, the assumption $A\cap B\neq \{1\}$ implies that $\langle A, B\rangle$ is also an $\X$-subgroup of $G$. Similarly,  a group $G$ is a $\CSX$-group if  all of its maximal $\X$-subgroups  are malnormal. The  general theory of $\XT$- and $\CSX$-groups is developed in \cite{Omar-Shah2} and the reader will find the necessary definition and results there. 

Gildenhuys, Kharlampovich, and Myasnikov obtained conditions under which HNN-extensions and amalgamated free products of $\mathrm{CSA}$-groups are again $\mathrm{CSA}$ (see \cite{Gilden}). In this article,  their results are generalized to free constructions of $\CSX$-groups. We prove the following two results:\\

{\bf Theorem A}. Suppose  $A$ and $B$ are $\CSX$-groups without involutions and $C\leq A, B$ is malnormal in both $A$ and $B$. Then, the amalgamated free product $A\ast_{_C} B$ is $\CSX$. \\

{\bf Theorem B}. Let $\CSX\subseteq \XT$ and $G$ be a $\CSX$-group without any involution. Suppose $A$ and $B$ are two malnormal and separated subgroups of $G$ and $\alpha:A\to B$ is an isomorphism. Then, the HNN-extension 
$$
G^{\ast}_{\alpha}=\langle G, t:\ t^{-1}xt=\alpha(x),\ x\in A\rangle 
$$
is $\CSX$. \\

{\em Separated} in the above theorem means that $A\cap B^g=\{ 1\}$ for all $g\in G$. Note that, in both theorems, we may assume that $\X$ is a proper variety because otherwise the results are trivially true.

For clarity, we introduce the following notation. The subgroup generated by a subset $X$ of a group $G$ will be denoted by $\langle X\rangle$, and a conjugate $a^x$ (or $A^x$) is $x^{-1}ax$ (similarly, $x^{-1}Ax$). A commutator $[x, y]$ is $x^{-1}y^{-1}xy$, but a similar notation $[u, v]$ will be used for the geodesic connecting a vertex $u$ to a  vertex $v$, in the case of trees. The amalgamated free product of two groups $A$ and $B$ with the amalgam $C$ will be denoted by $A\ast_{_C} B$, and the HNN-extension of a group $G$ with the associated subgroups $A$ and $B$ and the isomorphism $\alpha:A\to B$ will be denoted by $G^{\ast}_{\alpha}$. Our arguments are based on Bass-Serre theory and as such, the reader may wish to consult \cite{Dicks} and \cite{Serre}. Before proving the theorems A and B above, we review the (undirected) {\em Bass-Serre trees} of the aforementioned free constructions.

Let $G=A\ast_{_C} B$. Define an undirected graph $T$ as follows: $V_T$, the set of vertices, is the collection of all left cosets of $A$ and $B$ in $G$ i.e.,
$$
V_T=\{ gA, gB:\ g\in G\}.
$$
The set of edges, $E_T$, is the collection of all pairs $\{ gA, gB\}$ with $g$ in $G$, in other words
$$
E_T=\{ \{ gA, gB\}:\ g\in G\}.
$$
According to the main theorem of Bass-Serre theory, this graph is in fact a tree. The group $G$ acts on it by  left multiplication and the vertex and edge stabilizers are
$$
\Stab_G(gA)=gAg^{-1},\ \Stab_G(gB)=gBg^{-1},\ \Stab_G(\{ gA, gB\})=gCg^{-1}.
$$
Similarly, for the HNN-extension 
$$
\Gamma=G^{\ast}_{\alpha}=\langle G, t:\ t^{-1}xt=\alpha(x),\ x\in A\rangle,
$$
the Bass-Serre tree $T$ consists of the set of vertices
$$
V_T=\{ gG:\ g\in \Gamma\}
$$
and the set of edges
$$
E_T=\{ \{gG, gtG\}:\ g\in \Gamma\}.
$$
The group $\Gamma$ acts on $T$ by  left multiplication and we have 
$$
\Stab_{\Gamma}(gG)=gGg^{-1}.
$$
If we consider the special vertex $v^{\ast}=G$ and the special edges $e^{\ast}=\{ G, tG\}$ and $e_{\ast}=\{ G, t^{-1}G\}$, then $V_T=\mathrm{Orb}_{\Gamma}(v^{\ast})$ and $E_T=\mathrm{Orb}_{\Gamma}(e^{\ast})=\mathrm{Orb}_{\Gamma}(e_{\ast})$. In the proof of Theorem B, the stabilizers of edges will be needed. Therefore, we must determine the stabilizers of $e^{\ast}$ and $e_{\ast}$, first.

Suppose $g\in \Gamma$ fixes $e^{\ast}$. As $e^{\ast}$ has no direction, there are two possibilities; either
$$
gG=G\quad \mathrm{and}\quad gtG=tG
$$
or 
$$
gG=tG\quad \mathrm{and}\quad gtG=G.
$$
The second case is impossible by Britton's lemma: $gG=tG$ implies that $g=tg_0$ for some $g_0\in G$, while $gtG=G$ implies $gt\in G$. Hence we must have $tg_0t\in G$, violating Britton's lemma. This shows that 
$$
\Stab_{\Gamma}(e^{\ast})=\{ g\in \Gamma:\ gG=G, \ gtG=tG\}=A.
$$ 
Similarly, we have $\Stab_{\Gamma}(e_{\ast})=B$. Now, suppose $e=\{ xG, xtG\}$ is an arbitrary edge incident to the vertex $v^{\ast}=G$. Then there will be two cases:\\

{\bf $A$-type edge}. In this case, we have $xG=G$, so $x\in G$ and hence $e=\{ G, xtG\}=xe^{\ast}$. For such an edge we have 
$$
\Stab_{\Gamma}(e)=x\Stab_{\Gamma}(e^{\ast})x^{-1}=xAx^{-1}.
$$

{\bf $B$-type edge}. In this case, we have $xtG=G$, so $x=yt^{-1}$ for some $y\in G$. Therefore $e=\{ G, yt^{-1}\}=ye_{\ast}$, and consequently, 
$\Stab_{\Gamma}(e)=yBy^{-1}$. The stabilizers of other edges can be determined using translation. 

\section{Proof of Theorem A}
Let $T$ be the corresponding Bass-Serre tree of $G=A\ast_{_C} B$. We begin with the following lemma.

\begin{lemma}\label{fix-1}
Let $g\in G$ be a non-identity element. Then $g$ fixes at most one edge of $T$. 
\end{lemma}

\begin{proof}
Suppose $e_1$ and $e_2$ are two distinct edges left fixed by $g$. Let $p=e_1f_1\ldots f_ke_2$ be the unique geodesic which connects one end of $e_1$ to one end of $e_2$. As $g$ fixes $e_1$ and $e_2$ and preserves adjacency, it must fix the whole of $p$. So, we may assume that $e_1$  and $e_2$ are adjacent. After a suitable conjugation, we may assume that $e_1=\{ A, B\}$ and $e_2=\{ A, aB\}$, for some $a\in A\setminus C$. Hence, we have
$$
g\in \Stab_G(e_1)\cap \Stab_G(e_2)=C\cap aCa^{-1}.
$$
But, $C$ is malnormal in $A$ and consequently, $g=1$, a contradiction. 
\end{proof}

Now, suppose $M$ is a maximal $\X$-subgroup of $G$. We show that $M$ is malnormal in $G$. Consider the restriction of the action on $M$. There are two cases:\\

{\bf Case 1}. Suppose $M$ fixes some vertex of $T$, say $v=gA$. This implies that $M\subseteq gAg^{-1}$. As $A$ is $\CSX$ and $M$ is a maximal $\X$-subgroup of $gAg^{-1}$ as well, it is malnormal in $gAg^{-1}$. Let $x\in G$ and $1\neq h\in M\cap M^x$. Then $h$ stabilizes the vertices $gA$ and $x^{-1}gA$. Suppose $gA\neq x^{-1}gA$ and consider the geodesic $p$ connecting $gA$ to $x^{-1}gA$. Note that $p$ contains at least two edges and this together with Lemma \ref{fix-1} implies that $h=1$, a contradiction. So $gA=x^{-1}gA$ and therefore, $x\in gAg^{-1}$. Hence if $M\cap M^x\neq \{ 1\}$, then $x\in M$, proving that $M$ is malnormal in $G$.\\

{\bf Case 2}. Suppose $M$ has no fixed vertex. We claim that there is an element $h\in M$ which has no fixed vertex. Suppose on the contrary that every element of $M$ has a fixed point. Let $\Fix_T(h)$ denote the set of all fixed points of $h$. Note that $\Fix_T(h)$ is always a subtree; if $u, v\in \Fix_T(h)$, then $p$, the geodesic connecting $u$ to $v$, will be left fixed by $h$. Now, let $h_1, h_2\in M$. We claim that
$$
\Fix_T(h_1)\cap \Fix_T(h_2)\neq \emptyset. 
$$
This claim is true because of the following lemma which is well-known and the reader can find a proof in \cite{Sahatt}. 

\begin{lemma}\label{Sahatt}
Let $T$ be a tree and $h_1, h_2$ be two automorphisms of $T$. Suppose $h_1$ and $h_2$ have fixed vertices but $\Fix_T(h_1)\cap \Fix_T(h_2)=\emptyset$. Then $h_1h_2$ has no fixed vertex. 
\end{lemma}

Recall that for every $h\neq 1$, at most one edge is left fixed by $h$, so $|\Fix_T(h)|\leq 2$. Hence
$$
\Fix_T(M)=\bigcap_{h\in M}\Fix_T(h)
$$
is non-empty, a contradiction. As a result, $M$ contains an element $h$ which has no fixed vertex. We show that $M$ is an infinite cyclic group. Note that corresponding to the {\em hyperbolic} element $h$, there is a line (the axis of $h$)
$$
L=\{ \ldots, v_{-1}, v_0, v_1, v_2, \ldots\}
$$
and an integer $r\geq 1$, such that for all $i$ we have $hv_i=v_{i+r}$. So, $h$ acts as a translation on $L$. 

\begin{lemma}\label{inv-1}
The line $L$ is invariant under the action of $M$.  
\end{lemma}

\begin{proof}
Let $g\in M$ and consider the line $gL$. Note that $gL$ is the axis of $h^{g^{-1}}$. There are three possibilities:\\

{\bf i}. Suppose $gL\cap L$ is a segment, say $p=[u, v]$, for some vertices $u, v\in V_T$. Define four rays as follows: Order $L$ and $gL$ linearly by $v_i<v_j$ and $gv_i<gv_j$ iff $i<j$. Assume that $u<v$. Then
$$
L^+=]v, +\infty[=\{ v_i:\ v<v_i\}, \ L^-=]-\infty, u[=\{ v_i:\ v_i<u\}
$$
and define $(gL)^+$ and $(gL)^-$, similarly. 

\begin{figure}[ht]
\centering
\begin{tikzpicture}[
    line width=0.8pt,
    line cap=round,
    line join=round
]

\coordinate (P) at (0,0);
\coordinate (Q) at (3.2,0);

\draw[densely dashed]
    (-2.15,1.85)
    .. controls (-1.95,1.55) and (-1.65,1.12)
    .. (-1.30,0.85);

\draw
    (-1.30,0.85)
    .. controls (-0.85,0.45) and (-0.35,0.05)
    .. (P);

\draw[densely dashed]
    (-2.15,-0.82)
    .. controls (-1.95,-0.65) and (-1.65,-0.42)
    .. (-1.32,-0.27);

\draw
    (-1.32,-0.27)
    .. controls (-0.90,-0.10) and (-0.38,-0.02)
    .. (P);

\draw
    (P)
    .. controls (0.85,-0.07) and (2.35,0.07)
    .. (Q);

\draw
    (Q)
    .. controls (3.72,0.04) and (4.15,0.42)
    .. (4.62,0.82);

\draw[densely dashed]
    (4.62,0.82)
    .. controls (4.83,1.02) and (5.02,1.19)
    .. (5.22,1.34);

\draw
    (Q)
    .. controls (3.72,-0.08) and (4.12,-0.55)
    .. (4.57,-0.83);

\draw[densely dashed]
    (4.57,-0.83)
    .. controls (4.82,-0.94) and (5.02,-1.02)
    .. (5.22,-1.08);

\fill (P) circle (1.7pt);
\fill (Q) circle (1.7pt);

\node at (0.20,0.20) {$u$};
\node at (3,0.20) {$v$};

\node at (-1.65,1.63) {$L^{-}$};
\node at (4.85,1.53) {$L^{+}$};

\node at (-1.20,-0.65) {$(gL)^{-}$};
\node at (4.15,-1.27) {$(gL)^{+}$};

\node at (1.60,-0.27) {$p$};

\end{tikzpicture}
\caption{The intersection of the lines $L$ and $gL$.}
\label{fig:intersection-lines}
\end{figure}
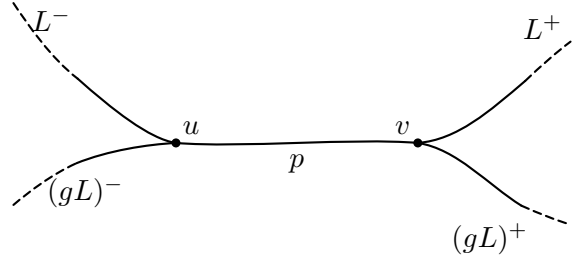

Let $a=h^m$ and $b=gh^ng^{-1}$ for sufficiently large $m$ and $n$. Note that both $a$ and $b$ have infinite orders. Define four subtrees as follows: $T^+_a$ is the connected component of the graph $T\setminus p$ which contains $L^+$, and $T^-_a$ is the component which contains $L^-$. Similarly, define $T^+_b$ and $T^-_b$ as the connected components of $T\setminus p$ which contain $(gL)^+$ and $(gL)^-$, respectively. Clearly, these subtrees are disjoint.
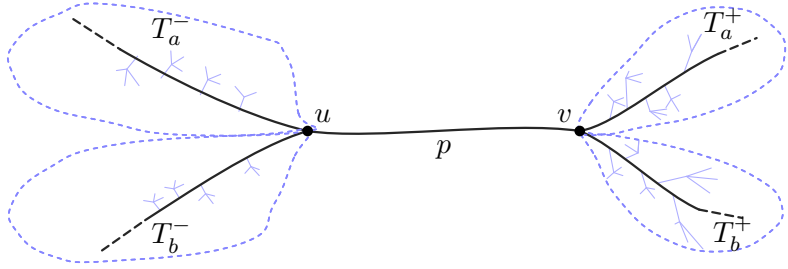
\begin{figure}[htbp]
\centering
\begin{tikzpicture}[
    main/.style={
        draw=black!85,
        line width=0.85pt,
        line cap=round,
        line join=round
    },
    twig/.style={
        draw=blue!32!white,
        line width=0.4pt,
        line cap=round,
        line join=round
    },
    boundary/.style={
        draw=blue!48!white,
        line width=0.65pt,
        dash pattern=on 1.4pt off 2.1pt,
        line cap=round
    }
]


\draw[boundary]
    plot[smooth cycle, tension=0.8] coordinates {
        (0,0)
        (-0.12,0.23)
        (-0.28,0.99)
        (-0.77,1.28)
        (-2.55,1.67)
        (-3.34,1.53)
        (-3.87,1.17)
        (-3.91,0.78)
        (-3.53,0.36)
        (-2.65,0.12)
        (-1.67,-0.05)
    };

\draw[boundary]
    plot[smooth cycle, tension=0.8] coordinates {
        (0,0)
        (-0.18,-0.35)
        (-0.44,-0.96)
        (-0.88,-1.49)
        (-2.61,-1.72)
        (-3.39,-1.65)
        (-3.84,-1.34)
        (-3.91,-0.96)
        (-3.46,-0.51)
        (-2.54,-0.18)
        (-0.55,-0.04)
    };

\draw[boundary]
    plot[smooth cycle, tension=0.8] coordinates {
        (3.6,0)
        (3.74,0.37)
        (4.25,0.94)
        (4.84,1.48)
        (5.60,1.63)
        (6.16,1.38)
        (6.31,1.00)
        (6.05,0.61)
        (5.26,0.23)
        (4.37,0.04)
        (3.98,-0.04)
    };

\draw[boundary]
    plot[smooth cycle, tension=0.8] coordinates {
        (3.6,0)
        (3.80,-0.33)
        (4.11,-0.90)
        (4.73,-1.43)
        (5.52,-1.59)
        (6.11,-1.34)
        (6.27,-0.98)
        (5.92,-0.61)
        (5.09,-0.28)
        (4.17,-0.08)
        (3.78,-0.04)
    };


\begin{scope}
\clip plot[smooth cycle, tension=0.8] coordinates {
     (0,0)
        (-0.12,0.23)
        (-0.28,0.99)
        (-0.77,1.28)
        (-2.55,1.67)
        (-3.34,1.53)
        (-3.87,1.17)
        (-3.91,0.78)
        (-3.53,0.36)
        (-2.65,0.12)
        (-1.67,-0.05)
};

\draw[twig] (-2.53,0.80) -- (-2.39,0.82) -- (-2.23,0.60);
\draw[twig] (-2.39,0.82) -- (-2.42,0.60);
\draw[twig] (-2.39,0.82) -- (-2.25,0.99);

\draw[twig] (-1.94,0.70) -- (-1.80,0.88) -- (-1.81,1.06);
\draw[twig] (-1.80,0.88) -- (-1.64,0.98);
\draw[twig] (-1.80,0.88) -- (-1.66,0.79);

\draw[twig] (-1.42,0.48) -- (-1.32,0.68) -- (-1.15,0.75);
\draw[twig] (-1.32,0.68) -- (-1.35,0.85);
\draw[twig] (-1.32,0.68) -- (-1.16,0.62);

\draw[twig] (-0.91,0.26) -- (-0.84,0.45) -- (-0.96,0.59);
\draw[twig] (-0.84,0.45) -- (-0.68,0.51);

\end{scope}

\begin{scope}
\clip plot[smooth cycle, tension=0.8] coordinates {
    (0,0)
        (-0.18,-0.35)
        (-0.44,-0.96)
        (-0.88,-1.49)
        (-2.61,-1.72)
        (-3.39,-1.65)
        (-3.84,-1.34)
        (-3.91,-0.96)
        (-3.46,-0.51)
        (-2.54,-0.18)
        (-0.55,-0.04)
};

\draw[twig]
    (-1.951,-1.051) -- (-2.04,-0.94);
\draw[twig]
    (-2.04,-0.94) -- (-2.17,-0.92);
\draw[twig]
    (-2.04,-0.94) -- (-2.02,-0.83);
\draw[twig]
    (-2.04,-0.94) -- (-1.91,-0.88);

\draw[twig]
    (-1.697,-0.887) -- (-1.77,-0.765);
\draw[twig]
    (-1.77,-0.765) -- (-1.89,-0.745);
\draw[twig]
    (-1.77,-0.765) -- (-1.75,-0.655);
\draw[twig]
    (-1.77,-0.765) -- (-1.64,-0.705);

\draw[twig]
    (-1.415,-0.709) -- (-1.35,-0.83);
\draw[twig]
    (-1.35,-0.83) -- (-1.44,-0.92);
\draw[twig]
    (-1.35,-0.83) -- (-1.23,-0.86);
\draw[twig]
    (-1.35,-0.83) -- (-1.25,-0.96);

\draw[twig]
    (-0.807,-0.353) -- (-0.75,-0.47);
\draw[twig]
    (-0.75,-0.47) -- (-0.83,-0.56);
\draw[twig]
    (-0.75,-0.47) -- (-0.64,-0.51);
\draw[twig]
    (-0.75,-0.47) -- (-0.64,-0.59);

\end{scope}

\begin{scope}
\clip plot[smooth cycle, tension=0.8] coordinates {
   (3.6,0)
        (3.74,0.37)
        (4.25,0.94)
        (4.84,1.48)
        (5.60,1.63)
        (6.16,1.38)
        (6.31,1.00)
        (6.05,0.61)
        (5.26,0.23)
        (4.37,0.04)
        (3.98,-0.04)
};

\draw[twig]
    (3.997,0.151) -- (4.05,0.39);
\draw[twig]
    (4.05,0.39) -- (3.97,0.48);
\draw[twig]
    (4.05,0.39) -- (4.18,0.44);
\draw[twig]
    (4.05,0.39) -- (4.20,0.33);

\draw[twig]
    (4.228,0.284) -- (4.21,0.70);
\draw[twig]
    (4.21,0.70) -- (4.25,0.79);
\draw[twig]
    (4.21,0.70) -- (4.43,0.76);
\draw[twig]
    (4.21,0.70) -- (4.41,0.64);

\draw[twig]
    (4.470,0.440) -- (4.53,0.26);
\draw[twig]
    (4.53,0.26) -- (4.67,0.15);
\draw[twig]
    (4.53,0.26) -- (4.75,0.21);
\draw[twig]
    (4.53,0.26) -- (4.75,0.29);

\draw[twig]
    (4.721,0.606) -- (4.81,0.42);
\draw[twig]
    (4.81,0.42) -- (4.74,0.31);
\draw[twig]
    (4.81,0.42) -- (4.93,0.48);
\draw[twig]
    (4.81,0.42) -- (4.91,0.26);

\draw[twig]
    (4.977,0.770) -- (5.08,1.05);
\draw[twig]
    (5.08,1.05) -- (5.03,1.27);
\draw[twig]
    (5.08,1.05) -- (5.21,1.28);
\draw[twig]
    (5.08,1.05) -- (5.20,1.27);

\end{scope}

\begin{scope}
\clip plot[smooth cycle, tension=0.8] coordinates {
    (3.6,0) (3.80,-0.33) (4.11,-0.90)
    (4.73,-1.43) (5.52,-1.59) (6.11,-1.34)
    (6.27,-0.98) (5.92,-0.61) (5.09,-0.28)
    (4.17,-0.08) (3.78,-0.04)
};

\draw[twig]
    (3.986,-0.217) -- (4.05,-0.55);
\draw[twig]
    (4.05,-0.55) -- (3.96,-0.63);
\draw[twig]
    (4.05,-0.55) -- (4.17,-0.59);
\draw[twig]
    (4.05,-0.55) -- (4.15,-0.48);

\draw[twig]
    (4.197,-0.368) -- (4.37,-0.29);
\draw[twig]
    (4.37,-0.29) -- (4.20,-0.18);
\draw[twig]
    (4.37,-0.29) -- (4.38,-0.13);
\draw[twig]
    (4.37,-0.29) -- (4.39,-0.02);

\draw[twig]
    (4.410,-0.529) -- (4.48,-0.75);
\draw[twig]
    (4.48,-0.75) -- (4.40,-0.83);
\draw[twig]
    (4.48,-0.75) -- (4.59,-0.89);
\draw[twig]
    (4.48,-0.75) -- (4.59,-0.77);

\draw[twig]
    (4.623,-0.690) -- (5.01,-0.60);
\draw[twig]
    (5.01,-0.60) -- (5.34,-0.49);
\draw[twig]
    (5.01,-0.60) -- (5.42,-0.64);
\draw[twig]
    (5.01,-0.60) -- (5.32,-0.82);

\draw[twig]
    (4.831,-0.839) -- (4.92,-1.2);
\draw[twig]
    (4.92,-1.2) -- (4.86,-1.23);
\draw[twig]
    (4.92,-1.2) -- (5.03,-1.47);
\draw[twig]
    (4.92,-1.2) -- (5.23,-1.56);
\end{scope}


\draw[main]
    (-2.48,1.08)
    .. controls (-1.95,0.77) and (-0.88,0.20)
    .. (0,0);

\draw[main,densely dashed]
    (-3.10,1.48) -- (-2.48,1.08);

\draw[main]
    (-2.15,-1.18)
    .. controls (-1.65,-0.86) and (-0.67,-0.19)
    .. (0,0);

\draw[main,densely dashed]
    (-2.72,-1.58) -- (-2.15,-1.18);

\draw[main]
    (3.6,0)
    .. controls (4.08,0.08) and (4.78,0.72)
    .. (5.43,1.02);

\draw[main,densely dashed]
    (5.43,1.02) -- (5.94,1.23);

\draw[main]
    (3.6,0)
    .. controls (4.09,-0.20) and (4.70,-0.83)
    .. (5.18,-1.04);

\draw[main,densely dashed]
    (5.18,-1.04) -- (5.75,-1.15);


\draw[main]
    (0,0)
    .. controls (0.85,-0.13) and (2.66,0.14)
    .. (3.6,0);


\draw[boundary]
    plot[smooth cycle, tension=0.8] coordinates {
         (0,0)
        (-0.12,0.23)
        (-0.28,0.99)
        (-0.77,1.28)
        (-2.55,1.67)
        (-3.34,1.53)
        (-3.87,1.17)
        (-3.91,0.78)
        (-3.53,0.36)
        (-2.65,0.12)
        (-1.67,-0.05)
    };

\draw[boundary]
    plot[smooth cycle, tension=0.8] coordinates {
        (0,0)
        (-0.18,-0.35)
        (-0.44,-0.96)
        (-0.88,-1.49)
        (-2.61,-1.72)
        (-3.39,-1.65)
        (-3.84,-1.34)
        (-3.91,-0.96)
        (-3.46,-0.51)
        (-2.54,-0.18)
        (-0.55,-0.04)
    };

\draw[boundary]
    plot[smooth cycle, tension=0.8] coordinates {
        (3.6,0)
        (3.74,0.37)
        (4.25,0.94)
        (4.84,1.48)
        (5.60,1.63)
        (6.16,1.38)
        (6.31,1.00)
        (6.05,0.61)
        (5.26,0.23)
        (4.37,0.04)
        (3.98,-0.04)
    };

\draw[boundary]
    plot[smooth cycle, tension=0.8] coordinates {
        (3.6,0) (3.80,-0.33) (4.11,-0.90)
        (4.73,-1.43) (5.52,-1.59) (6.11,-1.34)
        (6.27,-0.98) (5.92,-0.61) (5.09,-0.28)
        (4.17,-0.08) (3.78,-0.04)
    };


\fill (0,0) circle (2pt);
\fill (3.6,0) circle (2pt);

\node at (0.20,0.20) {$u$};
\node at (3.40,0.20) {$v$};

\node at (-1.80,1.34) {$T^-_a$};
\node at (-1.80,-1.38) {$T^-_b$};

\node at (5.50,1.42) {$T^+_a$};
\node at (5.63,-1.36) {$T^+_b$};

\node at (1.80,-0.25) {$p$};

\end{tikzpicture}
\caption{The four connected components.}
\label{fig:two-lines-common-segment}
\end{figure}

Now, for every $k>0$, we have $a^k(T\setminus T^-_a)\subseteq T^+_a$. To see why, let $w$ be a vertex such that $w\not\in T^-_a$. Let $\pi_L(w)$ be the projection of $w$ on $L$, i.e. the closest vertex of $L$ to $w$. Then we have $\pi_L(a^kw)=a^k\pi_L(w)$. On the other hand, we must have $\pi_L(w)\in p\cup L^+$. Hence $a^kw\in T^+_a$ for sufficiently large $m$ depending only on the length of $p$. Similarly, we have $a^{-k}(T\setminus T^+_a)\subseteq T^-_a$, as well as the inclusions
$$
b^k(T\setminus T^-_b)\subseteq T^+_b, \ b^{-k}(T\setminus T^+_b)\subseteq T^-_b
$$
for sufficiently large $n$. By the ping-pong lemma, we have $\langle a, b\rangle \cong \mathbb{F}_2$, the free group of rank $2$. Therefore $\mathbb{F}_2\in \X$, violating the assumption that $\X$ was a proper variety. \\

{\bf ii}. Suppose $gL\cap L=\emptyset$. Let $p=[u, v]$ be the shortest path connecting $L$ and $gL$. 

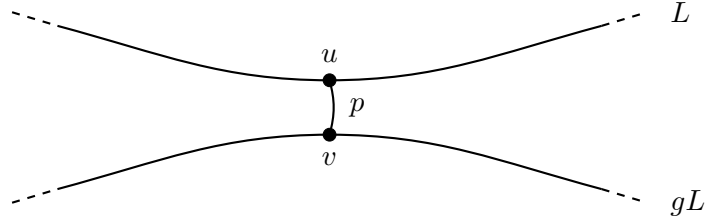
\begin{figure}[htbp]
  \centering
  \begin{tikzpicture}[scale=1.2]
    \draw[thick] (-3, 0.9) to[out=-15, in=180] (0, 0.3) to[out=0, in=195] (3, 0.9);
    \draw[thick, dashed] (-3.5, 1.05) -- (-3, 0.9);
    \draw[thick, dashed] (3, 0.9) -- (3.5, 1.05);
    \node[right, xshift=5pt] at (3.5, 1.05) {$L$};

    \draw[thick] (-3, -0.9) to[out=15, in=180] (0, -0.3) to[out=0, in=165] (3, -0.9);
    \draw[thick, dashed] (-3.5, -1.05) -- (-3, -0.9);
    \draw[thick, dashed] (3, -0.9) -- (3.5, -1.05);
    \node[right, xshift=5pt] at (3.5, -1.05) {$gL$};

    \filldraw (0, 0.3) circle (2pt) node[above=3pt] {$u$};
    \filldraw (0, -0.3) circle (2pt) node[below=3pt] {$v$};

    \draw[thick] (0, 0.3) to[out=-75, in=75] node[right=2pt] {$p$} (0, -0.3);
  \end{tikzpicture}
  \caption{The case of no common vertex}
\end{figure}

Again, define $L^+$, $L^-$, $(gL)^+$, and $(gL)^-$ as before: $L^+=]u, +\infty[$, $L^-=]-\infty, u[$, $(gL)^+=]v, +\infty[$, and $(gL)^-=]-\infty, v[$. Suppose $a=h^m$ and $b=gh^ng^{-1}$ for sufficiently large $m$ and $n$.  Define the subtrees $T^+_a$, $T^-_a$, $T^+_b$, and $T^-_b$ as in the previous case. One similarly verifies that for every $k>0$
$$
a^k(T\setminus T^-_a)\subseteq T^+_a, \ a^{-k}(T\setminus T^+_a)\subseteq T^-_a
$$
and 
$$
b^k(T\setminus T^-_b)\subseteq T^+_b, \ b^{-k}(T\setminus T^+_b)\subseteq T^-_b.
$$
Hence the ping-pong lemma implies that $\langle a, b\rangle \cong \mathbb{F}_2$, another contradiction.\\

{\bf iii}. Suppose $gL\cap L$ is infinite. Hence, it must be an infinite ray $R$. Suppose again that $a=h^m$ and $b=gh^ng^{-1}$ for sufficiently large $m$ and $n$. Let $w$ be an arbitrary vertex of $R$, sufficiently far from the first vertex of $R$.

\begin{figure}[htbp]
  \centering
  \begin{tikzpicture}[scale=1.2]
    \draw[thick] (-3, 1.0) to[out=-35, in=180] (0, 0) to[out=0, in=200] (3, 0.4);
    \draw[thick, dashed] (-3.4, 1.2) -- (-3, 1.0);
    \node[left, yshift=3pt] at (-3.4, 1.2) {$L$};

    \draw[thick] (-2.5, -0.8) to[out=30, in=180] (0, 0);
    \draw[thick, dashed] (-2.9, -1.0) -- (-2.5, -0.8);
    \node[left, yshift=-3pt] at (-2.9, -1.0) {$gL$};

    \node[below right] at (3, 0.4) {$R$};

    \filldraw (1.5, 0.06) circle (2pt) node[above=3pt] {$w$};
  \end{tikzpicture}
  \caption{The case of infinite ray.}
\end{figure}
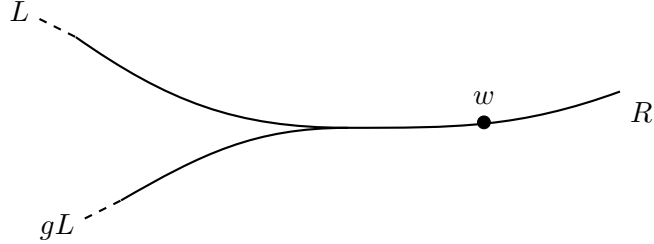

We have $[a, b]w=w$ so $[a, b]$ fixes all edges of $R$, and hence, Lemma \ref{fix-1} implies that $[a, b]=1$. This means that $a$ and $b$ have the same axes and consequently $gL=L$. 

Therefore, we have proved that $L$ is invariant under the action of $M$.
\end{proof}

Note that the action of $M$ on $L$ is faithful. Hence, $M$ embeds in the automorphism group $\mathrm{Aut}(L)$. Recall that every automorphism of $L$ is a translation or a reflection. But $G$ contains no involution, so every element of $M$ acts as a translation on $L$ and consequently, $M$ is an infinite cyclic group. Now, we are ready to show that $M$ is malnormal in $G$. Note that $M$ is the setwise stabilizer of $L$. To see why, let 
$$
H=\{ g\in G:\ gL=L\}.
$$
Then every element of $H$ acts on $L$ as a translation and again we conclude that $H$ is infinite cyclic. But $M\subseteq H$, so the maximality of $M$ gives $M=H$. Let $1\neq h\in M\cap M^g$ for some $g\in G$. As $h\in M$, the line $L$ is its axis. At the same time, $h\in g^{-1}Mg$ and hence, $g^{-1}L$ is its axis. Therefore, $gL=L$ which shows that $h$ belongs to the setwise stabilizer of $L$ which is $M$.

\section{Proof of Theorem B}
Let $\Gamma=G^{\ast}_{\alpha}$ and $T$ be the Bass-Serre tree of $\Gamma$. We start with the following analogue of Lemma \ref{fix-1}.

\begin{lemma}\label{fix-2}
Let $g\in \Gamma$ be a non-identity element. Then $g$ fixes at most one edge of $T$.
\end{lemma}

\begin{proof}
Suppose $g$ fixes two distinct edges $e_1$ and $e_2$. As in the proof of Lemma \ref{fix-1}, we may assume that $e_1$ and $e_2$ share  $v^{\ast}=G$ as a common vertex. Suppose $e_1$ is $A$-type, say $e_1=\{ G, x_1tG\}$ for some $x_1\in G$. Then,  there are two possibilities for $e_2$:\\

{\bf i}. Either $e_2$ is also $A$-type, say $e_2=\{ G, x_2tG\}$ for some $x_2\in G$. Then 
$$
g\in \Stab_{\Gamma}(e_1)\cap \Stab_{\Gamma}(e_2)=(x_1Ax_1^{-1})\cap (x_2Ax_2^{-1})
$$
and this shows that $A\cap A^{x_2^{-1}x_1}\neq \{ 1\}$, violating malnormality of $A$ in $G$. \\

{\bf ii}. Or $e_2$ is $B$-type, say $e_2=\{ G, y_1t^{-1}G\}$ for some $y_1\in G$. In this case we have
$$
g\in (x_1Ax_1^{-1})\cap (y_1By_1^{-1})
$$
which violates the assumption of separation. 
\end{proof}

\begin{lemma}
The subgroups $A$ and $B$ are malnormal in $\Gamma$. 
\end{lemma}

\begin{proof}
Let $1\neq b\in B$ and $c\in \Gamma$ be an arbitrary element which does not belong to $G$. Suppose $c^{-1}bc\in G$. We prove that $c\in Bt^{-1}G$. Consider a reduced form of $c$:
$$
c=g_0t^{\varepsilon_1}g_1t^{\varepsilon_2}\cdots g_{n-1}t^{\varepsilon_n}g_n
$$
where $g_i\in G$ and $\varepsilon_i=\pm 1$. Then we have
$$
c^{-1}bc=g_n^{-1}t^{-\varepsilon_n}\cdots t^{-\varepsilon_1}g_0^{-1}bg_0t^{\varepsilon_1}\cdots t^{\varepsilon_n}g_n\in G.
$$
So, there must be a cancelation in the middle piece $t^{-\varepsilon_1}g_0^{-1}bg_0t^{\varepsilon_1}$. There are two cases:\\

{\bf i}. Let $\varepsilon_1=1$. In this case cancelation starts only if $g_0^{-1}bg_0\in A$. Hence $1\neq b\in B\cap g_0Ag_0^{-1}$ and this is impossible by the separation assumption.\\

{\bf ii}. Let $\varepsilon_1=-1$. Then we must have $g_0^{-1}bg_0\in B$ and the malnormality of $B$ in $G$ implies $g_0\in B$. Therefore $c=b_0t^{-1}c_1$ where $b_0\in B$ and $c_1\in \Gamma$. Put $a_0=tb_0^{-1}bb_0t^{-1}$ which is an element of $A$. Since $b\neq 1$, we have $a_0\neq 1$. The assumption $c^{-1}bc\in G$ becomes $c_1^{-1}a_0c_1\in G$. Suppose $c_1$ does not belong to $G$. Then $c_1=g_1t^{\varepsilon_2}\cdots t^{\varepsilon_n}g_n$ and once more, there are two possibilities:\\

{\bf ii-a}. If $\varepsilon_2=1$, then Britton's lemma implies $g_1^{-1}a_0g_1\in A$ in which case malnormality of $A$ in $G$ forces $g_1\in A$. But $t^{-1}g_1t$ is a piece of the reduced form of $c$; a contradiction.\\

{\bf ii-b}. If $\varepsilon_2=-1$, then $g_1^{-1}a_0g_1\in B$ and so $a_0$ belongs to the intersection of $A$ and $g_1Bg_1^{-1}$, violating separation of $A$ and $B$.\\

Hence $c_1\in G$, and consequently, $c\in Bt^{-1}G$. Now, we can prove that $B$ is malnormal in $\Gamma$. Let $c\in \Gamma\setminus G$ and $1\neq b\in B\cap cBc^{-1}$ with $b=c^{-1}b_1c$ where $b_1\in B$. Then $c=b_0t^{-1}g$ for some $b_0\in B$ and $g\in G$, as we just proved. Hence
$$
b=c^{-1}b_1c=g^{-1}tb_0^{-1}bb_0t^{-1}g\in A^g.
$$
Therefore, $b\in B\cap A^g$, violating the separation assumption. This contradiction shows that $c\in G$ and then the malnormality of $B$ in $G$ forces $c\in B$. 
\end{proof}

Now, define a homomorphism $\varepsilon: \Gamma \to \mathbb{Z}$ by extending the map $\varepsilon (x)=0$ for all $x\in G$ and $\varepsilon (t)=1$. One can verify that the kernel of $\varepsilon$ is
$$
K=\langle t^{-n}Gt^n:\ n\in \mathbb{Z}\rangle.
$$
We can describe $K$ as a chain of iterated amalgamated free products: Let $K_0=G$ and for $i>0$ define
$$
K_i=(t^iGt^{-i})\ast_{_{E_{-i}=B}}G_{i-1}\ast_{_{E_i=B}}(t^{-i}Gt^i).
$$
Here $E_i=t^{-i}At^i$, a copy of $A$ in $t^{-i}Gt^i$,  $B\subseteq G_{i-1}$, and the identification is $t^{-i}xt^i\mapsto \alpha(x)$. Then we have $K=\bigcup_{i=0}^{\infty}K_i$. Note that by Theorem A each $K_i$ is $\CSX$ and as $\CSX\subseteq\XT$, using Proposition 1.8 of \cite{Omar-Shah2}, we conclude that $K$ is $\CSX$.

\begin{lemma}
Let $M$ be a maximal $\X$-subgroup of $K$ and $g\in \Gamma$. If $M\cap M^g\neq \{ 1\}$, then $g\in K$.
\end{lemma}

\begin{proof}
As $K$ is a normal subgroup of $\Gamma$, we have $M^g\subseteq K$, so both $M$ and $M^g$ are maximal $\X$-subgroups of the $\CSX$-group $K$. But $\CSX\subseteq \XT$, hence $\langle M, M^g\rangle$ must be an $\X$-subgroup, and therefore, the maximality implies $M^g=M$. Now, suppose on the contrary that $g\not\in K$. If $g$ has a fixed vertex, then $g\in G^x\subseteq K$ (for some element $x$) and this shows that $g$ has no fixed vertex. Since $g$ normalizes $M$, for every $1\neq h\in M$ and every integer $j$, we have $h^{g^j}\in M$. Recall that $\Fix_T(h)$ is finite. We assume that this set is non-empty and we choose $j$ big enough to ensure that
$$
\Fix_T(h)\cap g^{-j}\Fix_T(h)=\emptyset.
$$
But, $\Fix_T(h^{g^j})=g^{-j}\Fix_T(h)$ and hence
$$
\Fix_T(h)\cap \Fix_T(h^{g^j})=\emptyset.
$$
Now, using Lemma \ref{Sahatt} once again, the element $h^{\prime}=hh^{g^j}$ is hyperbolic. So, we can proceed as in the proof of Lemma \ref{inv-1} to show that $M$ is an infinite cyclic group. Once this is proved, one can use the argument of of the last part of the current proof to arrive at a contradiction $g\in K$.

Consequently, every non-identity element of $M$ is hyperbolic. Let $h_1, h_2\in M$ be two non-identity elements and suppose $L_1$ and $L_2$ are their corresponding axes, respectively. We prove that $L_1=L_2$. Consider the intersection $L_1\cap L_2$. There are three cases and for each these cases, one can argue as in the proof of Lemma \ref{inv-1}:\\

{\bf i}. If $L_1\cap L_2$ is a finite segment, then a suitable application of the ping-pong lemma implies that $\mathbb{F}_2\in \X$. \\

{\bf ii}. If $L_1\cap L_2=\emptyset$ then,  again using the ping-pong lemma we see that $\mathbb{F}_2\in \X$. \\

{\bf iii}. If $L_1\cap L_2$ contains an infinite ray, then as in the proof of Lemma \ref{inv-1}, we deduce that $[h_1, h_2]=1$ and therefore $L_1=L_2$.\\

Let $L=L_1$. Each element of $M$ acts on $L$ as a translation and hence $M$ embeds into $\mathrm{Aut}(L)$; as in the proof of Theorem A, the subgroup $M$ will be an infinite cyclic group. Let $h$ be a generator of $M$ and return to  $M^g=M$. The element $h^g$ has the same axis $L$, whilst on the other hand, its axis must be $g^{-1}L$. Consequently, $gL=L$ and as $g$ is hyperbolic, it is a translation of $L$. Consider the setwise stabilizer of $L$:
$$
H=\{ x\in \Gamma:\ xL=L \}. 
$$
Then $H$ embeds into $\mathrm{Aut}(L)$ and hence $H=\langle a\rangle$ is cyclic as well. We have $M\subseteq H$ and $g\in H$. So $h=a^i$ and $g=a^j$, for some integers $i$ and $j$. As $h\in M$, we have $\varepsilon(h)=0$ and hence $\varepsilon(a)=0$. This implies that $\varepsilon(g)=j\varepsilon(a)=0$ and therefore, $g\in K$.
\end{proof}

Now, we are ready to prove that $M$ is malnormal in $\Gamma$. There are two cases:\\

{\bf i}. Suppose $M$ has a fixed vertex. Then we may assume that $M\subseteq G\subseteq K$. So, $M$ is a maximal $\X$-subgroup of $K$, as well. Let $g\in \Gamma$ and $M\cap M^g\neq \{ 1\}$. Then by the above lemma, $g\in K$ and as $K$ is $\CSX$, we must have $g\in M$.\\

{\bf ii}. Suppose $M$ has no fixed vertex. Then as in the proof of Theorem A, we conclude that there is a hyperbolic element in $M$ with the axis $L$ and  this axis is invariant under the action of $M$. The same argument as in the last paragraph of the proof of Theorem A shows that $M$ is malnormal in $\Gamma$.

\section{An application}
In \cite{Omar-Shah3}, we studied split Frobenius groups which are equational domains or $\XT$-groups. The proof of Theorem B above provides us an example of a split Frobenius group with a torsion-free complement: Consider the HNN-extension $\Gamma=G^{\ast}_{\alpha}$ of the previous section. Note that $\Gamma=\langle t\rangle \ltimes K$ and as we proved, $\Gamma$ is a $\CSX$-group. We show that the action of $H=\langle t\rangle$ on $K$ is fixed-point-free and consequently $\Gamma$ is a Frobenius group. We must show that for all non-zero integers $n$, the centralizer $C_K(t^n)$ is trivial. Define a bi-infinite sequence of vertices in the tree $T$ as follows:
$$
u_0=G, \ u_i=t^iu_0=t^iG.
$$
Then $L$ below is a line and it is the axis of all non-identity elements $t^n$:
$$
L=\{ \ldots, u_{-1}, u_0, u_1, u_2, \ldots \}.
$$
Now, suppose $n\neq 0$ and $k\in K$. Suppose $t^{-n}kt^n=k$. Then $k$ commutes with $t^n$ and hence $kL=L$. There are two cases:\\

{\bf i}. If $k$ has a fixed vertex then,  it cannot be a hyperbolic element and hence the action of $k$ on $L$ is not a translation. Therefore, it has a fixed point, say $u_i$. Then for any integer $m$, we have
$$
k(t^{mn}u_i)=t^{mn}ku_i=t^{mn}u_i,
$$
and hence $k$ fixes an  infinite line, showing that $k=1$. \\

{\bf ii}. The element $k$ is hyperbolic and so, $L$ is the axis of $k$. Hence, $k$ acts as a translation of $L$ i.e. there is an integer $d$ such that $ku_i=u_{i+d}$ for all $i$. Note that $t^d$ is also a translation of $L$ by $d$ edges. Hence $k=t^d$ and consequently $0=\varepsilon(k)=d\varepsilon(t)=d$. This shows that $k=1$. \\

Summarizing, $\Gamma$ is a split Frobenius group with a torsion-free complement which belongs to the class $\CSX$.

\end{document}